\documentclass[letterpaper,english]{extarticle}
\usepackage[T1]{fontenc}
\usepackage[latin9]{inputenc}
\usepackage{color}
\usepackage{babel}
\usepackage{mathrsfs}
\usepackage{amsmath}
\usepackage{amsthm}
\usepackage{amssymb}
\usepackage{graphicx}
\usepackage{esint}
\usepackage[unicode=true,pdfusetitle,
 bookmarks=true,bookmarksnumbered=false,bookmarksopen=false,
 breaklinks=false,pdfborder={0 0 0},pdfborderstyle={},backref=false,colorlinks=true]
 {hyperref}

\makeatletter

\providecommand{\tabularnewline}{\\}

\theoremstyle{plain}
\newtheorem{lem}{\protect\lemmaname}
\theoremstyle{plain}
\newtheorem{thm}{\protect\theoremname}

\@ifundefined{date}{}{\date{}}

\makeatletter

\newcommand{\executeiffilenewer}[3]{%
\ifnum\pdfstrcmp{\pdffilemoddate{#1}}%
{\pdffilemoddate{#2}}>0%
{\immediate\write18{#3}}\fi%
}
\newcommand{\includesvg}[2][scale=1]{%
\executeiffilenewer{#2.svg}{#2.pdf}%
{inkscape -z -D --file=#2.svg --export-pdf=#2.pdf}%
\includegraphics[#1]{#2.pdf}%
}

\providecommand{\tabularnewline}{\\}

\usepackage{dsfont}\usepackage{babel}
\usepackage{relsize}
\usepackage{subdepth}
\allowdisplaybreaks

\usepackage{color}

\usepackage{babel}

\providecommand{\lemmaname}{Lemma}

\providecommand{\theoremname}{Theorem}

\usepackage{babel}
\providecommand{\lemmaname}{Lemma}
  
\providecommand{\theoremname}{Theorem}

\usepackage{url}

\usepackage{enumitem}

\newtheorem{assumption}{Assumption}

\usepackage{dsfont}

\DeclareFontFamily{OT1}{pzc}{}
\DeclareFontShape{OT1}{pzc}{m}{it}{<-> s * [1.200] pzcmi7t}{}
\DeclareMathAlphabet{\mathpzc}{OT1}{pzc}{m}{it}

\makeatother

\usepackage{babel}
\providecommand{\lemmaname}{Lemma}
\providecommand{\theoremname}{Theorem}

\definecolor{airforceblue}{rgb}{0.36, 0.54, 0.66}
\definecolor{ballblue}{rgb}{0.13, 0.67, 0.8}
\hypersetup{citecolor=ballblue}

\definecolor{alizarin}{rgb}{0.82, 0.1, 0.26}
\definecolor{asparagus}{rgb}{0.53, 0.66, 0.42}
\definecolor{applegreen}{rgb}{0.55, 0.71, 0.0}
\definecolor{armygreen}{rgb}{0.29, 0.33, 0.13}
\definecolor{amber(sae/ece)}{rgb}{1.0, 0.49, 0.0}
\definecolor{coquelicot}{rgb}{1.0, 0.22, 0.0}
\definecolor{ao(english)}{rgb}{0.0, 0.5, 0.0}

\renewcommand\footnotemark{}

\makeatother

\providecommand{\lemmaname}{Lemma}
\providecommand{\theoremname}{Theorem}

\begin{document}
\title{\textbf{Risk-Aware MMSE Estimation}}
\author{Dionysios S. Kalogerias, Luiz F. O. Chamon,\\
George J. Pappas, and Alejandro Ribeiro\thanks{Contact (e-mail): \{dionysis, luizf, pappasg, aribeiro\}@seas.upenn.edu.}\\
$\,$\\
\textit{\normalsize{}Department of Electrical \& Systems Engineering}\\
\textit{\normalsize{}University of Pennsylvania}}
\maketitle
\begin{abstract}
Despite the simplicity and intuitive interpretation of Minimum Mean
Squared Error (MMSE) estimators, their effectiveness in certain scenarios
is questionable. Indeed, minimizing squared errors on average does
not provide any form of stability, as the volatility of the estimation
error is left unconstrained. When this volatility is statistically
significant, the difference between the average and realized performance
of the MMSE estimator can be drastically different. To address this
issue, we introduce a new risk-aware MMSE formulation which trades
between mean performance and risk by explicitly constraining the expected
predictive variance of the involved squared error. We show that, under
mild moment boundedness conditions, the corresponding risk-aware optimal
solution can be evaluated explicitly, and has the form of an appropriately
biased nonlinear MMSE estimator. We further illustrate the effectiveness
of our approach via several numerical examples, which also showcase
the advantages of risk-aware MMSE estimation against risk-neutral
MMSE estimation, especially in models involving skewed, heavy-tailed
distributions.
\end{abstract}
\hypersetup{linkcolor=ao(english)}\hypersetup{linkcolor=red}

\section{Introduction}

Critical applications require that stochastic decisions be made not
only on the basis of minimizing average losses, but also safeguarding
against less probable, though possibly catastrophic, events. Examples
appear naturally in many areas, including wireless industrial control
\cite{Ahlen2019}, energy \cite{Bruno2016,Moazeni2015}, finance \cite{Markowitz1952,Follmer2002,Shang2018},
robotics \cite{Kim2019,Pereira2013}, LIDAR \cite{Bedi2019}, and
networking \cite{Ma2018}. In such cases, the ultimate goal is to
obtain \textit{risk-aware decision policies} that hedge against statistically
significant extreme losses, even at the cost of slightly sacrificing
performance under nominal operating conditions.
\begin{figure}
\centering\hspace{-1bp}\includegraphics[scale=0.39]{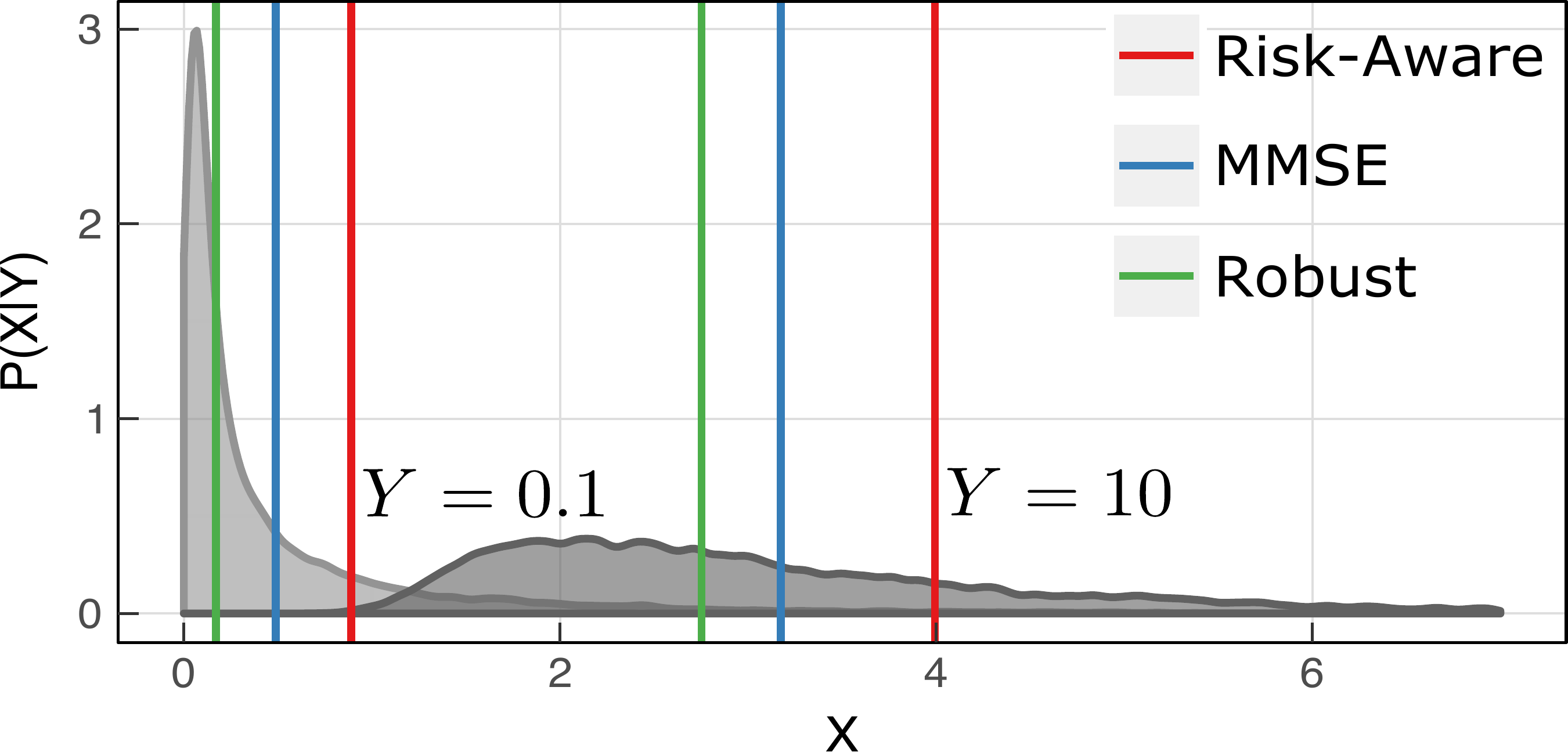}\caption{Comparison between risk-neutral and risk-aware estimates.}
\end{figure}

To illustrate this effect, consider the problem of estimating a random
state~$X$ from observations corrupted by state-dependent noise,
namely $Y|X\hspace{-0.5pt}\hspace{-0.5pt}\hspace{-0.5pt}\hspace{-0.5pt}\hspace{-0.5pt}\sim\hspace{-0.5pt}\hspace{-0.5pt}\hspace{-0.5pt}{\cal N}\hspace{-0.5pt}(X,\hspace{-0.5pt}9\hspace{-0.3pt}X^{2})$
(see Section 5). In such a setting, either small or large values of
$Y$ provide highly ambiguous evidence, since, in both cases, they
can come from either small or large values of $X$. This is corroborated
by Fig. 1, which displays the posterior distribution ${\cal P}_{X|Y}$
for two values of $Y$. While the MMSE estimator may incur severe
losses, the risk-aware estimator we develop in this work -- shown
as red vertical lines in Fig. 1 -- hedges against observation ambiguity,
therefore avoiding extreme prediction errors. Such behavior is achieved
by \textit{biasing estimates} towards the tail of the posterior ${\cal P}_{X|Y}$,
by the right amount, for each realization of $Y$. Although the risk-aware
estimator may incur larger losses on average, it performs statistically
more consistently across realizations of $Y$ (also see risk curve
in Fig. 2, Section 5). It is also worth contrasting risk-awareness
with statistical robustness, whose goal is to protect against deviations
from a nominal model. Robust estimators -- green vertical lines in
Fig. 1 -- promote insensitivity to tail events, which they designate
as statistically insignificant. On the other hand, estimators resulting
from risk-aware formulations treat these events as statistically significant,
though relatively infrequent (see Table 1).

Over the last three decades, risk-aware optimization has grown increasingly
popular and has been studied in the contexts of both decision making
and learning \cite{A.2018,Cardoso2019,W.Huang2017,Jiang2017,Kalogerias2018b,Tamar2017,Vitt2018,Zhou2018}.
In risk-aware optimization, expectations are replaced by more general
functionals, called \textit{risk measures} \cite{ShapiroLectures_2ND},
whose purpose is to quantify the statistical volatility of random
losses, as well as mean performance. Popular examples include mean-variance
functionals \cite{Markowitz1952,ShapiroLectures_2ND}, mean-semideviations
\cite{Kalogerias2018b}, and Conditional Value-at-Risk (CVaR) \cite{Rockafellar1997}.

In Bayesian estimation, risk awareness is typically achieved by replacing
the classical quadratic cost with its exponentiation~\cite{Whittle1981,Speyer1992,Moore1997,Dey1997,Dey1999}.
However, although sometimes effective, this approach is not without
limitations. First, the need for finiteness of the moment generating
function of the quadratic cost excludes heavy-tailed distributions,
which are precisely those that incur high risk. Second, the exponential
approach does not provide an interpretable way to control the trade-off
between mean performance and risk, making it hard to use in settings
where explicit risk levels must be met. Third, it does not result
in a simple, general solution as in classical MMSE estimation, challenging
its practical applicability. Finally, it does not effectively quantify
\textit{observation-induced risk}, inherent in problems where measurements
provide ambiguous evidence.
\begin{table}
\centering%
\begin{tabular}{|c||c|c|}
\cline{2-3} \cline{3-3} 
\multicolumn{1}{c|||}{\textbf{\textit{Uncertainty}}} & Frequent  & Infrequent\tabularnewline
\hline 
\hline 
Significant  & Model  & Risk\tabularnewline
\hline 
Insignificant  & Noise  & Outliers\tabularnewline
\hline 
\end{tabular}\caption{Classification of statistical uncertainty.}
\end{table}

In this work, we pose risk-aware \textit{functional} Bayesian estimation
as a constrained MMSE problem, where squared errors are minimized
on average, \textit{subject to} a bound on their expected conditional
variance. We show that, under mild conditions, this formulation results
in a convex variational problem that admits a closed-form solution.
The resulting \textit{optimal risk-aware nonlinear MMSE estimator}
is applicable to a wide variety of generative models, including highly
skewed and/or heavy-tailed distributions. The effectiveness of our
approach is confirmed via numerical examples, also demonstrating its
advantages against risk-neutral MMSE estimation.

\section{Problem Formulation }

Let $(\Omega,\mathscr{F},{\cal P})$ be a probability space, and consider
an arbitrary pair of random elements $\boldsymbol{X}:\Omega\rightarrow\mathbb{R}^{M}$
and $\boldsymbol{Y}:\Omega\rightarrow\mathbb{R}^{N}$ on $(\Omega,\mathscr{F})$.
We are interested in the problem of estimating $\boldsymbol{X}$ from
a \textit{single} realization of $\boldsymbol{Y}$ in a Bayesian setting,
namely by assuming knowledge of the joint probability distribution~${\cal P}_{(\boldsymbol{X},\boldsymbol{Y})}$.
We may conveniently think of $\boldsymbol{Y}$ as available \textit{observations},
on the basis of which we would like to make predictions about the
\textit{hidden state}~$\boldsymbol{X}$. Undoubtedly, this general
problem is fundamental in many areas, including statistics, signal
processing, machine learning, and control, and with numerous interesting
applications.

Of course, an established approach to the prediction problem considered
is to choose an estimator $\widehat{\boldsymbol{X}}\hspace{-1pt}\hspace{-1pt}\hspace{-1pt}:\hspace{-1pt}\hspace{-1pt}\hspace{-1pt}\Omega\hspace{-1pt}\hspace{-1pt}\hspace{-1pt}\rightarrow\hspace{-1pt}\hspace{-1pt}\hspace{-1pt}\mathbb{R}^{M}$
as a solution to the \textit{stochastic variational }MMSE program
\global\long\def\arraystretch{1.3}%
 
\begin{equation}
\begin{array}{rl}
\underset{\widehat{\boldsymbol{X}}:\Omega\rightarrow\mathbb{R}^{M}}{\mathrm{minimize}} & \mathbb{E}\{\Vert\boldsymbol{X}-\widehat{\boldsymbol{X}}\Vert_{2}^{2}\}\\
\mathrm{subject\,to} & \widehat{\boldsymbol{X}}\text{ is }\mathscr{Y}\text{-measurable}
\end{array},\label{eq:MMSE_Problem}
\end{equation}
where $\mathscr{Y}\hspace{-1pt}\hspace{-1pt}\hspace{-0.5pt}\equiv\hspace{-1pt}\hspace{-1pt}\sigma\hspace{-0.5pt}\{\boldsymbol{Y}\}$
denotes the sub-$\sigma$-algebra of $\mathscr{F}$ generated by $\boldsymbol{Y}$.
Problem (\ref{eq:MMSE_Problem}) is well-understood under rather general
conditions. In fact, if we \textit{merely} assume that $\boldsymbol{X}\hspace{-0.5pt}\in\hspace{-0.5pt}{\cal L}_{1}(\Omega,\mathscr{Y},{\cal P};\mathbb{R}^{M})\equiv{\cal L}_{1|\mathscr{Y}}^{M}$,
an optimal solution to (\ref{eq:MMSE_Problem}) is given by any conditional
expectation of $\boldsymbol{X}$ relative to $\boldsymbol{Y}$, i.e.,
$\widehat{\boldsymbol{X}}^{*}(\boldsymbol{Y})\equiv\mathbb{E}\{\boldsymbol{X}|\boldsymbol{Y}\}$.


However, despite the simplicity of MMSE estimation, as well as its
intuitive geometric interpretation in Hilbert space whenever $\boldsymbol{X}\in{\cal L}_{2|\mathscr{Y}}^{M}$,
its effectiveness is often questionable. Indeed, minimizing the squared
error $\Vert\boldsymbol{X}-\widehat{\boldsymbol{X}}\Vert_{2}^{2}$
\textit{in expectation} does \textit{not} provide stability or robustness,
in the sense that statistically significant variability of the resulting
\textit{optimal prediction error} is uncontrolled. In other words,
the MMSE problem (\ref{eq:MMSE_Problem}) is \textit{risk-neutral}.
This has important consequences from a practical perspective, since
the error realization~$\Vert\boldsymbol{X}-\widehat{\boldsymbol{X}}^{*}\hspace{-0.5pt}\hspace{-0.5pt}\hspace{-0.5pt}(\boldsymbol{Y})\Vert_{2}^{2}$
experienced in practice may be far from the expected value $\mathbb{E}\{\Vert\boldsymbol{X}-\widehat{\boldsymbol{X}}^{*}\hspace{-0.5pt}\hspace{-0.5pt}\hspace{-0.5pt}(\boldsymbol{Y})\Vert_{2}^{2}\}$,
or even the predictive statistic $\mathbb{E}\{\Vert\boldsymbol{X}-\widehat{\boldsymbol{X}}^{*}\hspace{-0.5pt}\hspace{-0.5pt}\hspace{-0.5pt}(\boldsymbol{Y})\Vert_{2}^{2}|\boldsymbol{Y}\}$.
It is then clear that achieving small error variability is at least
as desirable as achieving minimal errors on average.


Motivated by the previous discussion, we consider a nontrivial variation
of the risk-neutral MMSE problem (\ref{eq:MMSE_Problem}), striking
a \textit{balance between mean performance and risk}. Specifically,
we introduce and study the \textit{constrained} stochastic variational
problem 
\begin{equation}
\begin{array}{rl}
\underset{\widehat{\boldsymbol{X}}:\Omega\rightarrow\mathbb{R}^{M}}{\mathrm{minimize}} & \mathbb{E}\{\Vert\boldsymbol{X}-\widehat{\boldsymbol{X}}\Vert_{2}^{2}\}\\
\mathrm{subject\,to} & \mathbb{E}\{\mathbb{V}_{\boldsymbol{Y}}\{\Vert\boldsymbol{X}-\widehat{\boldsymbol{X}}\Vert_{2}^{2}\}\hspace{-1pt}\}\le\varepsilon\\
 & \widehat{\boldsymbol{X}}\text{ is }\mathscr{Y}\text{-measurable}
\end{array},\label{eq:Base_Problem-V}
\end{equation}
where 
\begin{equation}
\mathbb{V}_{\boldsymbol{Y}}\{\Vert\boldsymbol{X}-\widehat{\boldsymbol{X}}\Vert_{2}^{2}\}\hspace{-1pt}\hspace{-0.5pt}\triangleq\hspace{-0.5pt}\hspace{-1pt}\mathbb{E}\big\{\hspace{-0.5pt}\hspace{-1pt}\big(\hspace{-0.5pt}\Vert\boldsymbol{X}-\widehat{\boldsymbol{X}}\Vert_{2}^{2}-\mathbb{E}\{\Vert\boldsymbol{X}-\widehat{\boldsymbol{X}}\Vert_{2}^{2}|\boldsymbol{Y}\}\big)^{2}|\boldsymbol{Y}\hspace{-1pt}\big\}
\end{equation}
is the predictive variance of $\Vert\boldsymbol{X}-\widehat{\boldsymbol{X}}\Vert_{2}^{2}$
relative to $\boldsymbol{Y}$, and $\varepsilon>0$ is a fixed risk
tolerance. In words, problem (\ref{eq:Base_Problem-V}) constrains
the \textit{expected predictive variance of the quadratic cost} $\Vert\boldsymbol{X}-\widehat{\boldsymbol{X}}\Vert_{2}^{2}$,
known in the statistics literature as the \textit{unexplained component}
of its variance; the latter is due to the law of total variance. In
other words, the constraint quantifies the uncertainty of MMSE-optimally
predicting the quadratic cost \textit{achieved} \textit{by choosing
an estimator }$\widehat{\boldsymbol{X}}(\boldsymbol{Y})$, on the
basis of the observations $\boldsymbol{Y}$. Of course, $\mathbb{E}\{\mathbb{V}_{\boldsymbol{Y}}\{\Vert\boldsymbol{X}-\widehat{\boldsymbol{X}}\Vert_{2}^{2}\}\hspace{-1pt}\}$
is a \textit{measure of risk}. Therefore, we suggestively refer to
the task fulfilled by problem (\ref{eq:Base_Problem-V}) as \textit{risk-aware
MMSE estimation}.

Problem (\ref{eq:Base_Problem-V}) confines the search for an optimal
MMSE estimator within the family of estimators exhibiting risk (in
the sense described above) within tolerance $\varepsilon$; thus,
problem (\ref{eq:Base_Problem-V}) is well-motivated. Naturally, an
optimal solution to the risk-aware problem (\ref{eq:Base_Problem-V})
in general achieves larger MSE as compared to the risk-neutral problem
(\ref{eq:MMSE_Problem}). However, the statistical variability of
the squared errors achieved by the former will be explicitly controlled,
according to the tunable tolerance $\varepsilon$, resulting in more
stable statistical prediction.

\section{Convex Variational QCQP Reformulation }

As it turns out, in such a general form, the risk-aware MMSE problem
(\ref{eq:Base_Problem-V}) is rather challenging to study, let alone
solve. Therefore, in the following, we will consider a slightly more
constrained version of (\ref{eq:Base_Problem-V}), by enforcing square
integrability on the decision $\widehat{\boldsymbol{X}}$, namely,
\begin{equation}
\begin{array}{rl}
\underset{\widehat{\boldsymbol{X}}:\Omega\rightarrow\mathbb{R}^{M}}{\mathrm{minimize}} & \mathbb{E}\{\Vert\boldsymbol{X}-\widehat{\boldsymbol{X}}\Vert_{2}^{2}\}\\
\mathrm{subject\,to} & \mathbb{E}\{\mathbb{V}_{\boldsymbol{Y}}\{\Vert\boldsymbol{X}-\widehat{\boldsymbol{X}}\Vert_{2}^{2}\}\hspace{-1pt}\}\le\varepsilon\\
 & \widehat{\boldsymbol{X}}\in{\cal L}_{2|\mathscr{Y}}^{M}
\end{array}.\label{eq:Base_Problem-L2}
\end{equation}
Of course, the additional ${\cal L}_{2}$ constraint in problem (\ref{eq:Base_Problem-L2})
may not be in favor of generality, per se, but it is harmless for
almost every practical consideration. Further, in the following we
make use of the following regularity condition on the statistical
behavior of $(\boldsymbol{X},\boldsymbol{Y})$.

\setlist[description]{style=multiline,leftmargin=20bp}

\noindent \begin{assumption}\label{AssumptionMain}It is true that
$\mathbb{E}\big\{\hspace{-1pt}\Vert\boldsymbol{X}\Vert_{2}^{3}|\boldsymbol{Y}\hspace{-1pt}\big\}\in{\cal L}_{2|\mathscr{Y}}^{1}$.\end{assumption}

In words, Assumption \ref{AssumptionMain} simply says that the third-order
moment filter $\mathbb{E}\big\{\hspace{-1pt}\Vert\boldsymbol{X}\Vert_{2}^{3}|\boldsymbol{Y}\hspace{-1pt}\big\}$
is of finite energy. Using Assumption \ref{AssumptionMain}, problem
(\ref{eq:Base_Problem-L2}) may be conveniently reformulated, as the
next result suggests. 
\begin{lem}
\textbf{\textup{(QCQP Reformulation of Problem (\ref{eq:Base_Problem-L2}))}}\label{lem:QCQP}
Suppose that Assumption \ref{AssumptionMain} is in effect, and define
the posterior covariance 
\begin{equation}
\boldsymbol{\Sigma}_{\boldsymbol{X}|\boldsymbol{Y}}\triangleq\mathbb{E}\big\{\hspace{-1pt}(\boldsymbol{X}-\mathbb{E}\{\boldsymbol{X}|\boldsymbol{Y}\})(\boldsymbol{X}-\mathbb{E}\{\boldsymbol{X}|\boldsymbol{Y}\})^{\boldsymbol{T}}|\boldsymbol{Y}\big\}\succeq{\bf 0}.\hspace{-1pt}
\end{equation}
Then, problem (\ref{eq:Base_Problem-L2}) is well-defined and equivalent
to the convex variational Quadratically Constrained Quadratic Program
(QCQP) 
\begin{equation}
\begin{array}{rl}
\hspace{-1pt}\hspace{-1pt}\hspace{-1pt}\hspace{-1pt}\hspace{-1pt}\hspace{-1pt}\hspace{-1pt}\hspace{-1pt}\hspace{-1pt}\hspace{-1pt}\underset{\widehat{\boldsymbol{X}}:\Omega\rightarrow\mathbb{R}^{M}}{\mathrm{minimize}} & \hspace{-1pt}\hspace{-1pt}\hspace{-1pt}\hspace{-1pt}\dfrac{1}{2}\mathbb{E}\big\{\Vert\widehat{\boldsymbol{X}}\Vert_{2}^{2}\hspace{-1pt}\hspace{-1pt}-\hspace{-0.5pt}\hspace{-1pt}2(\mathbb{E}\big\{\hspace{-1pt}\boldsymbol{X}|\boldsymbol{Y}\big\})^{\boldsymbol{T}}\widehat{\boldsymbol{X}}\hspace{-1pt}+\hspace{-1pt}\hspace{-1pt}\mathbb{E}\big\{\hspace{-1pt}\Vert\boldsymbol{X}\Vert_{2}^{2}|\boldsymbol{Y}\big\}\hspace{-1pt}\hspace{-0.5pt}\big\}\\
\hspace{-1pt}\hspace{-1pt}\hspace{-1pt}\hspace{-1pt}\hspace{-1pt}\hspace{-1pt}\hspace{-1pt}\hspace{-1pt}\hspace{-1pt}\hspace{-1pt}\mathrm{subject\,to} & \hspace{-1pt}\hspace{-1pt}\hspace{-1pt}\hspace{-1pt}\mathbb{E}\big\{\widehat{\boldsymbol{X}}^{\boldsymbol{T}}\boldsymbol{\Sigma}_{\boldsymbol{X}|\boldsymbol{Y}}\widehat{\boldsymbol{X}}-\big(\mathbb{E}\big\{\hspace{-1pt}\Vert\boldsymbol{X}\Vert_{2}^{2}\boldsymbol{X}|\boldsymbol{Y}\hspace{-1pt}\big\}\\
 & \hspace{-1pt}\hspace{-1pt}\hspace{-1pt}\hspace{-1pt}\quad-\mathbb{E}\big\{\hspace{-1pt}\Vert\boldsymbol{X}\Vert_{2}^{2}|\boldsymbol{Y}\hspace{-1pt}\big\}\mathbb{E}\{\boldsymbol{X}|\boldsymbol{Y}\}\big)^{\boldsymbol{T}}\widehat{\boldsymbol{X}}\big\}\\
 & \hspace{-1pt}\hspace{-1pt}\hspace{-1pt}\hspace{-1pt}\quad\quad\le\hspace{-1pt}\hspace{-1pt}\dfrac{\varepsilon-\mathbb{E}\big\{\hspace{-1pt}\mathbb{V}_{\boldsymbol{Y}}\{\Vert\boldsymbol{X}\Vert_{2}^{2}\}\hspace{-1pt}\hspace{-1pt}\big\}}{4}\\
 & \hspace{-1pt}\hspace{-1pt}\hspace{-1pt}\hspace{-1pt}\widehat{\boldsymbol{X}}\in{\cal L}_{2|\mathscr{Y}}^{M}
\end{array}\hspace{-1pt}\hspace{-1pt}\hspace{-1pt}\hspace{-1pt}\hspace{-1pt},\hspace{-1pt}\hspace{-1pt}\hspace{-1pt}\hspace{-1pt}\label{eq:QCQP}
\end{equation}
where all expectations and involved operations are well-defined. 
\end{lem}
\begin{proof}[Proof of Lemma \ref{lem:QCQP}]
We start with the objective of problem (\ref{eq:Base_Problem-L2}),
for which it is obviously true that
\begin{equation}
\mathbb{E}\{\Vert\boldsymbol{X}-\widehat{\boldsymbol{X}}\Vert_{2}^{2}\}\equiv\mathbb{E}\big\{\mathbb{E}\big\{\Vert\boldsymbol{X}\Vert_{2}^{2}-2\boldsymbol{X}^{\boldsymbol{T}}\widehat{\boldsymbol{X}}+\Vert\widehat{\boldsymbol{X}}\Vert_{2}^{2}|\boldsymbol{Y}\big\}\hspace{-1pt}\hspace{-0.5pt}\big\},
\end{equation}
since the expectation of $\Vert\boldsymbol{X}-\widehat{\boldsymbol{X}}\Vert_{2}^{2}$
always exists. Additionally, by invoking Cauchy-Schwarz twice, we
observe that
\begin{flalign}
\mathbb{E}\big\{|\boldsymbol{X}^{\boldsymbol{T}}\widehat{\boldsymbol{X}}|\big\} & \le\mathbb{E}\big\{\Vert\boldsymbol{X}\Vert_{2}\Vert\widehat{\boldsymbol{X}}\Vert_{2}\big\}\nonumber \\
 & \equiv\mathbb{E}\big\{\mathbb{E}\big\{\Vert\boldsymbol{X}\Vert_{2}|\boldsymbol{Y}\big\}\Vert\widehat{\boldsymbol{X}}\Vert_{2}\hspace{-1pt}\hspace{-0.5pt}\big\}\nonumber \\
 & \le\big\Vert\mathbb{E}\big\{\Vert\boldsymbol{X}\Vert_{2}|\boldsymbol{Y}\big\}\big\Vert_{{\cal L}_{2}}\big\Vert\Vert\widehat{\boldsymbol{X}}\Vert_{2}\big\Vert_{{\cal L}_{2}},\label{eq:Key_Bound}
\end{flalign}
where $\big\Vert\Vert\widehat{\boldsymbol{X}}\Vert_{2}\big\Vert_{{\cal L}_{2}}<\infty\iff\widehat{\boldsymbol{X}}\in{\cal L}_{2|\mathscr{Y}}^{M}$
by assumption, and Jensen implies that
\begin{align}
\big\Vert\mathbb{E}\big\{\Vert\boldsymbol{X}\Vert_{2}|\boldsymbol{Y}\big\}\big\Vert_{{\cal L}_{2}} & \le\big\Vert\mathbb{E}\big\{\Vert\boldsymbol{X}\Vert_{2}|\boldsymbol{Y}\big\}\big\Vert_{{\cal L}_{3}}\nonumber \\
 & \le\big(\mathbb{E}\big\{\hspace{-1pt}\hspace{-0.5pt}\big(\mathbb{E}\big\{\Vert\boldsymbol{X}\Vert_{2}^{3}|\boldsymbol{Y}\big\}\big)^{2\cdot1/2}\hspace{-1pt}\hspace{-0.5pt}\big\}\big)^{1/3}\nonumber \\
 & \le\big(\mathbb{E}\big\{\hspace{-1pt}\hspace{-0.5pt}\big(\mathbb{E}\big\{\Vert\boldsymbol{X}\Vert_{2}^{3}|\boldsymbol{Y}\big\}\big)^{2}\hspace{-1pt}\hspace{-0.5pt}\big\}\big)^{1/(2\cdot3)}\nonumber \\
 & \le\big\Vert\mathbb{E}\big\{\hspace{-1pt}\Vert\boldsymbol{X}\Vert_{2}^{3}|\boldsymbol{Y}\hspace{-1pt}\big\}\big\Vert_{{\cal L}_{2}}^{1/3}<\infty,
\end{align}
as well. Then $\mathbb{E}\big\{\boldsymbol{X}^{\boldsymbol{T}}\widehat{\boldsymbol{X}}\big\}$
is finite, and it follows that
\begin{equation}
\mathbb{E}\{\Vert\boldsymbol{X}-\widehat{\boldsymbol{X}}\Vert_{2}^{2}\}\equiv\mathbb{E}\big\{\Vert\widehat{\boldsymbol{X}}\Vert_{2}^{2}\hspace{-1pt}\hspace{-1pt}-\hspace{-0.5pt}\hspace{-1pt}2(\mathbb{E}\big\{\hspace{-1pt}\boldsymbol{X}|\boldsymbol{Y}\big\})^{\boldsymbol{T}}\widehat{\boldsymbol{X}}\hspace{-1pt}+\hspace{-1pt}\hspace{-1pt}\mathbb{E}\big\{\hspace{-1pt}\Vert\boldsymbol{X}\Vert_{2}^{2}|\boldsymbol{Y}\big\}\hspace{-1pt}\hspace{-0.5pt}\big\},
\end{equation}
as in the objective of (\ref{eq:QCQP}).

The constraint of (\ref{eq:Base_Problem-L2}) may be equivalently
reexpressed in a similar fashion, although the procedure is slightly
more involved. Specifically, by definition of $\mathbb{V}_{\boldsymbol{Y}}\{\Vert\boldsymbol{X}-\widehat{\boldsymbol{X}}\Vert_{2}^{2}\}$,
we may initially expand as
\begin{align}
 & \hspace{-1pt}\hspace{-1pt}\hspace{-1pt}\hspace{-1pt}\hspace{-1pt}\hspace{-1pt}\hspace{-1pt}\hspace{-1pt}\hspace{-1pt}\hspace{-1pt}\big(\hspace{-0.5pt}\Vert\boldsymbol{X}-\widehat{\boldsymbol{X}}\Vert_{2}^{2}-\mathbb{E}\{\Vert\boldsymbol{X}-\widehat{\boldsymbol{X}}\Vert_{2}^{2}|\boldsymbol{Y}\}\big)^{2}\nonumber \\
 & \equiv(\Vert\boldsymbol{X}\Vert_{2}^{2}-\mathbb{E}\{\Vert\boldsymbol{X}\Vert_{2}^{2}|\boldsymbol{Y}\})^{2}+4\widehat{\boldsymbol{X}}^{\boldsymbol{T}}(\boldsymbol{X}-\mathbb{E}\{\boldsymbol{X}|\boldsymbol{Y}\})(\boldsymbol{X}-\mathbb{E}\{\boldsymbol{X}|\boldsymbol{Y}\})^{\boldsymbol{T}}\widehat{\boldsymbol{X}}\nonumber \\
 & \quad-4\Vert\boldsymbol{X}\Vert_{2}^{2}\boldsymbol{X}^{\boldsymbol{T}}\widehat{\boldsymbol{X}}+4\Vert\boldsymbol{X}\Vert_{2}^{2}(\mathbb{E}\{\boldsymbol{X}|\boldsymbol{Y}\})^{\boldsymbol{T}}\widehat{\boldsymbol{X}}\nonumber \\
 & \quad\quad+4\mathbb{E}\{\Vert\boldsymbol{X}\Vert_{2}^{2}|\boldsymbol{Y}\}\boldsymbol{X}^{\boldsymbol{T}}\widehat{\boldsymbol{X}}-4\mathbb{E}\{\Vert\boldsymbol{X}\Vert_{2}^{2}|\boldsymbol{Y}\}(\mathbb{E}\{\boldsymbol{X}|\boldsymbol{Y}\})^{\boldsymbol{T}}\widehat{\boldsymbol{X}},\label{eq:expansion_1}
\end{align}
where the first two terms of the right-hand side of (\ref{eq:expansion_1})
are nonnegative. Consequently, it suffices to concentrate on the respective
last four dot product terms.

Using the same argument as in (\ref{eq:Key_Bound}), in order to show
that all these four terms have finite expectations, it suffices to
ensure that
\begin{equation}
\big\Vert\mathbb{E}\{\Vert\Vert\boldsymbol{X}\Vert_{2}^{2}\boldsymbol{X}\Vert_{2}|\boldsymbol{Y}\}\big\Vert_{{\cal L}_{2}}\equiv\big\Vert\mathbb{E}\{\Vert\boldsymbol{X}\Vert_{2}^{3}|\boldsymbol{Y}\}\big\Vert_{{\cal L}_{2}}<\infty,
\end{equation}
which is of course automatically true by Assumption \ref{AssumptionMain},
but also that
\begin{flalign}
\big\Vert\mathbb{E}\{\Vert\Vert\boldsymbol{X}\Vert_{2}^{2}\mathbb{E}\{\boldsymbol{X}|\boldsymbol{Y}\}\Vert_{2}|\boldsymbol{Y}\}\big\Vert_{{\cal L}_{2}} & \equiv\big\Vert\mathbb{E}\{\Vert\boldsymbol{X}\Vert_{2}^{2}\Vert\mathbb{E}\{\boldsymbol{X}|\boldsymbol{Y}\}\Vert_{2}|\boldsymbol{Y}\}\big\Vert_{{\cal L}_{2}}\nonumber \\
 & \equiv\big\Vert\Vert\mathbb{E}\{\boldsymbol{X}|\boldsymbol{Y}\}\Vert_{2}\mathbb{E}\{\Vert\boldsymbol{X}\Vert_{2}^{2}|\boldsymbol{Y}\}\big\Vert_{{\cal L}_{2}}\nonumber \\
 & <\infty,\\
\big\Vert\mathbb{E}\{\Vert\mathbb{E}\{\Vert\boldsymbol{X}\Vert_{2}^{2}|\boldsymbol{Y}\}\boldsymbol{X}\Vert_{2}|\boldsymbol{Y}\}\big\Vert_{{\cal L}_{2}} & \equiv\big\Vert\mathbb{E}\{\Vert\boldsymbol{X}\Vert_{2}|\boldsymbol{Y}\}\mathbb{E}\{\Vert\boldsymbol{X}\Vert_{2}^{2}|\boldsymbol{Y}\}\big\Vert_{{\cal L}_{2}}\nonumber \\
 & <\infty\quad\text{and}\\
\big\Vert\mathbb{E}\big\{\Vert\mathbb{E}\{\Vert\boldsymbol{X}\Vert_{2}^{2}|\boldsymbol{Y}\big\}\mathbb{E}\{\boldsymbol{X}|\boldsymbol{Y}\}\Vert_{2}|\boldsymbol{Y}\}\big\Vert_{{\cal L}_{2}} & \equiv\big\Vert\Vert\mathbb{E}\{\boldsymbol{X}|\boldsymbol{Y}\}\Vert_{2}\mathbb{E}\{\Vert\boldsymbol{X}\Vert_{2}^{2}|\boldsymbol{Y}\}\big\Vert_{{\cal L}_{2}}\\
 & <\infty.
\end{flalign}
Observe, though, that all three latter quantities are upper bounded
by the quantity $\big\Vert\mathbb{E}\{\Vert\boldsymbol{X}\Vert_{2}|\boldsymbol{Y}\}\mathbb{E}\{\Vert\boldsymbol{X}\Vert_{2}^{2}|\boldsymbol{Y}\}\big\Vert_{{\cal L}_{2}}$,
for which we may write (by Jensen) 
\begin{flalign}
\big\Vert\mathbb{E}\{\Vert\boldsymbol{X}\Vert_{2}|\boldsymbol{Y}\}\mathbb{E}\{\Vert\boldsymbol{X}\Vert_{2}^{2}|\boldsymbol{Y}\}\big\Vert_{{\cal L}_{2}}^{2} & \equiv\mathbb{E}\big\{(\mathbb{E}\{\Vert\boldsymbol{X}\Vert_{2}|\boldsymbol{Y}\})^{2}(\mathbb{E}\{\Vert\boldsymbol{X}\Vert_{2}^{2}|\boldsymbol{Y}\})^{2}\big\}\nonumber \\
 & \le\mathbb{E}\big\{\mathbb{E}\{\Vert\boldsymbol{X}\Vert_{2}^{2}|\boldsymbol{Y}\}(\mathbb{E}\{\Vert\boldsymbol{X}\Vert_{2}^{2}|\boldsymbol{Y}\})^{2}\big\}\nonumber \\
 & \equiv\mathbb{E}\big\{(\mathbb{E}\{\Vert\boldsymbol{X}\Vert_{2}^{2}|\boldsymbol{Y}\})^{3}\big\}\nonumber \\
 & \equiv\mathbb{E}\big\{(\mathbb{E}\{\Vert\boldsymbol{X}\Vert_{2}^{2}|\boldsymbol{Y}\})^{2\cdot3/2}\big\}\nonumber \\
 & \le\mathbb{E}\big\{(\mathbb{E}\{\Vert\boldsymbol{X}\Vert_{2}^{3}|\boldsymbol{Y}\})^{2}\big\}<\infty,
\end{flalign}
where the last line follows again by Assumption \ref{AssumptionMain}.

Given the discussion above, we may now take conditional expectations
on (\ref{eq:expansion_1}), to obtain the expression (note that all
operations involving conditional expectations are technically allowed
under our assumptions)
\begin{align}
 & \hspace{-1pt}\hspace{-1pt}\hspace{-1pt}\hspace{-1pt}\hspace{-1pt}\hspace{-1pt}\hspace{-1pt}\hspace{-1pt}\hspace{-1pt}\hspace{-1pt}\mathbb{V}_{\boldsymbol{Y}}\{\Vert\boldsymbol{X}-\widehat{\boldsymbol{X}}\Vert_{2}^{2}\}\nonumber \\
 & \equiv\mathbb{E}\big\{\hspace{-0.5pt}\hspace{-0.5pt}\hspace{-0.5pt}\big(\hspace{-0.5pt}\Vert\boldsymbol{X}-\widehat{\boldsymbol{X}}\Vert_{2}^{2}-\mathbb{E}\{\Vert\boldsymbol{X}-\widehat{\boldsymbol{X}}\Vert_{2}^{2}|\boldsymbol{Y}\}\big)^{2}|\boldsymbol{Y}\big\}\nonumber \\
 & \equiv\mathbb{V}_{\boldsymbol{Y}}\{\Vert\boldsymbol{X}\Vert_{2}^{2}\}+4\widehat{\boldsymbol{X}}^{\boldsymbol{T}}\boldsymbol{\Sigma}_{\boldsymbol{X}|\boldsymbol{Y}}\widehat{\boldsymbol{X}}\nonumber \\
 & \quad-4\big(\mathbb{E}\{\Vert\boldsymbol{X}\Vert_{2}^{2}\boldsymbol{X}|\boldsymbol{Y}\}-\mathbb{E}\{\Vert\boldsymbol{X}\Vert_{2}^{2}|\boldsymbol{Y}\}\mathbb{E}\{\boldsymbol{X}|\boldsymbol{Y}\}\big)^{\boldsymbol{T}}\widehat{\boldsymbol{X}}.\label{eq:E_expansion_1}
\end{align}
Taking expectations on both sides of (\ref{eq:E_expansion_1}) and
rearranging terms gives the desired expression for the constraint
of the QCQP (\ref{eq:QCQP}).
\end{proof}
Lemma \ref{lem:QCQP} is very useful, because it shows the equivalence
of problem (\ref{eq:Base_Problem-L2}) to the convex QCQP (\ref{eq:QCQP}),
which is well-defined and favorably structured. In particular, this
reformulation will allow us to effectively study problem (\ref{eq:Base_Problem-L2})
by looking at its variational Lagrangian dual. Actually, as we discuss
next, working in the dual domain will allow us to \textit{solve} problem
(\ref{eq:Base_Problem-L2}) in \textit{closed-form}. Of course, such
a closed form is important, not only because it provides an analytical,
textbook-level solution to a functional risk-aware problem, which
happens rather infrequently in such settings, but also because, as
we will see, the solution itself provides intuition, highlights connections
and enables comparison of problem (\ref{eq:Base_Problem-L2}) with
its risk-neutral counterpart (\ref{eq:MMSE_Problem}).

\section{Risk-Aware MMSE Estimators}

In our development, we exploit a variational version of Slater's condition,
which is one the most widely used constraint qualifications in both
deterministic and stochastic optimization.

\begin{assumption}\label{AssumptionMain2}Given $\varepsilon>0$,
problem (\ref{eq:Base_Problem-L2}) satisfies Slater's condition,
i.e., there exists $\widehat{\boldsymbol{X}}_{\dagger}\in{\cal L}_{2|\mathscr{Y}}^{M}$,
such that $\mathbb{E}\{\Vert\boldsymbol{X}-\widehat{\boldsymbol{X}}_{\dagger}\Vert_{2}^{2}\}<\infty$
and $\mathbb{E}\{\mathbb{V}_{\boldsymbol{Y}}\{\Vert\boldsymbol{X}-\widehat{\boldsymbol{X}}_{\dagger}\Vert_{2}^{2}\}\hspace{-1pt}\}<\varepsilon$.\end{assumption}

Under both Assumptions \ref{AssumptionMain} and \ref{AssumptionMain2},
it follows that the QCQP (\ref{eq:QCQP}) satisfies Slater's condition,
as well. Then, it must be the case that $\mathbb{E}\{\mathbb{V}_{\boldsymbol{Y}}\{\Vert\boldsymbol{X}\Vert_{2}^{2}\}\hspace{-1pt}\}\hspace{-1pt}\hspace{-1.5pt}<\hspace{-0.5pt}\hspace{-1pt}\infty$;
if not, Assumption \ref{AssumptionMain2} is impossible to hold. Further,
problem (\ref{eq:QCQP}) must be feasible, with convex effective domain\vspace{-1bp}
\begin{equation}
{\cal F}_{2|\mathscr{Y}}^{M}\hspace{-1pt}\hspace{-1pt}\hspace{-1pt}\triangleq\hspace{-1pt}\hspace{-1pt}\big\{\widehat{\boldsymbol{X}}\in{\cal L}_{2|\mathscr{Y}}^{M}\big|\mathbb{E}\big\{\widehat{\boldsymbol{X}}^{\boldsymbol{T}}\boldsymbol{\Sigma}_{\boldsymbol{X}|\boldsymbol{Y}}\widehat{\boldsymbol{X}}\big\}<\infty\big\}.\hspace{-1pt}\hspace{-1pt}
\end{equation}

Next, if Assumption \ref{AssumptionMain} holds, define the \textit{variational
Lagrangian }of the \textit{primal problem} (\ref{eq:QCQP}) $\mathpzc{L}\hspace{-0.5pt}\hspace{-0.5pt}\hspace{-0.5pt}:\hspace{-0.5pt}\hspace{-0.5pt}\hspace{-0.5pt}{\cal L}_{2|\mathscr{Y}}^{M}$
$\times\mathbb{R}_{+}\hspace{-0.5pt}\hspace{-0.5pt}\hspace{-0.5pt}\hspace{-0.5pt}\rightarrow\hspace{-0.5pt}\hspace{-0.5pt}\hspace{-0.5pt}(-\infty,\infty]$
as
\begin{flalign}
\mathpzc{L}\big(\widehat{\boldsymbol{X}},\mu\big) & \triangleq\dfrac{1}{2}\mathbb{E}\big\{\Vert\widehat{\boldsymbol{X}}\Vert_{2}^{2}-2(\mathbb{E}\big\{\hspace{-1pt}\boldsymbol{X}|\boldsymbol{Y}\big\})^{\boldsymbol{T}}\widehat{\boldsymbol{X}}\hspace{-1pt}+\hspace{-1pt}\hspace{-1pt}\mathbb{E}\big\{\hspace{-1pt}\Vert\boldsymbol{X}\Vert_{2}^{2}|\boldsymbol{Y}\big\}\hspace{-1pt}\hspace{-0.5pt}\big\}\nonumber \\
 & \quad+\mu\mathbb{E}\big\{\widehat{\boldsymbol{X}}^{\boldsymbol{T}}\boldsymbol{\Sigma}_{\boldsymbol{X}|\boldsymbol{Y}}\widehat{\boldsymbol{X}}-\big(\mathbb{E}\big\{\hspace{-1pt}\Vert\boldsymbol{X}\Vert_{2}^{2}\boldsymbol{X}|\boldsymbol{Y}\hspace{-1pt}\big\}\nonumber \\
 & \quad\quad-\mathbb{E}\big\{\hspace{-1pt}\Vert\boldsymbol{X}\Vert_{2}^{2}|\boldsymbol{Y}\hspace{-1pt}\big\}\mathbb{E}\{\boldsymbol{X}|\boldsymbol{Y}\}\big)^{\boldsymbol{T}}\widehat{\boldsymbol{X}}\big\}\nonumber \\
 & \quad\quad\quad-\mu\dfrac{\varepsilon-\mathbb{E}\big\{\hspace{-1pt}\mathbb{V}_{\boldsymbol{Y}}\{\Vert\boldsymbol{X}\Vert_{2}^{2}\}\hspace{-1pt}\hspace{-1pt}\big\}}{4},
\end{flalign}
where $\mu\in\mathbb{R}_{+}$ is a multiplier associated with the
constraint of (\ref{eq:QCQP}). The \textit{dual function} $\mathpzc{D}\hspace{-0.5pt}\hspace{-0.5pt}:\hspace{-0.5pt}\hspace{-0.5pt}\mathbb{R}_{+}\hspace{-0.5pt}\hspace{-0.5pt}\rightarrow\hspace{-0.5pt}\hspace{-0.5pt}\hspace{-0.5pt}(-\infty,\infty]$
is accordingly defined as
\begin{equation}
\mathpzc{D}(\mu)\triangleq\inf_{\widehat{\boldsymbol{X}}\in{\cal F}_{2|\mathscr{Y}}^{M}}\mathpzc{L}\big(\widehat{\boldsymbol{X}},\mu\big).
\end{equation}
If $\mathpzc{P}^{*}\hspace{-1pt}\hspace{-1pt}\in\hspace{-1pt}\hspace{-1pt}[0,\infty]$
denotes the optimal value of problem (\ref{eq:QCQP}), it is true
that $\mathpzc{D}\le\mathpzc{P}^{*}$ on $\mathbb{R}_{+}$. Then,
the optimal value of the always concave, \textit{dual problem}
\begin{equation}
\begin{array}{rl}
\mathrm{maximize} & \mathpzc{D}(\mu)\\
\mathrm{subject\,to} & \mu\ge0
\end{array},\label{eq:DUAL}
\end{equation}
defined as $\mathpzc{D}^{*}\hspace{-1pt}\hspace{-1pt}\hspace{-1pt}\triangleq\hspace{-1pt}\sup_{\mu\ge0}\hspace{-1pt}\hspace{-0.5pt}\mathpzc{D}(\mu)\hspace{-1pt}\hspace{-1pt}\hspace{-0.5pt}\in\hspace{-1pt}\hspace{-1pt}(-\infty,\hspace{-0.5pt}\infty]$,
is the tightest under-estimate of $\mathpzc{P}^{*}$, when knowing
only $\mathpzc{D}$.

Exploiting Assumptions \ref{AssumptionMain} and \ref{AssumptionMain2},
we may now formulate the following fundamental theorem, which establishes
that the convex variational problem (\ref{eq:QCQP}) exhibits zero
duality gap. This essentially follows as an application of standard
results in variational Lagrangian duality; see, for instance, (\cite{Luenberger1968},
Section 8.3, Theorem 1). The proof is therefore straightforward, and
omitted.
\begin{thm}
\textbf{\textup{(QCQP (\ref{eq:QCQP}): Zero Duality Gap)}}\label{thm:ZDG_Var}
Suppose that Assumptions \ref{AssumptionMain} and \ref{AssumptionMain2}
are in effect. Then, strong duality holds for problem (\ref{eq:QCQP}),
that is, $0\hspace{-1pt}\le\mathpzc{D}^{*}\equiv\mathpzc{P}^{*}<\infty$.
Additionally, the set of dual optimal solutions, $\arg\max_{\mu\ge0}\mathpzc{D}(\mu)$,
is nonempty. Further, if $\widehat{\boldsymbol{X}}_{*}$ is primal
optimal for (\ref{eq:QCQP}), it follows that $\widehat{\boldsymbol{X}}_{*}\equiv\widehat{\boldsymbol{X}}_{*}(\mu_{*})\in\arg\min_{\widehat{\boldsymbol{X}}\in{\cal F}_{2|\mathscr{Y}}^{M}}\mathpzc{L}\big(\widehat{\boldsymbol{X}},\mu_{*}\big)$,
where $0\le\mu_{*}\in\arg\max_{\mu\ge0}\mathpzc{D}(\mu)$.
\end{thm}
Leveraging Theorem \ref{thm:ZDG_Var}, it is possible to show that,
under Assumptions \ref{AssumptionMain} and \ref{AssumptionMain2},
the QCQP (\ref{eq:QCQP}) and, therefore, the original ${\cal L}_{2}$
risk-aware MMSE problem (\ref{eq:Base_Problem-L2}), admit a common
closed form solution. In this respect, we have the next theorem, which
constitutes the main result of this paper.
\begin{thm}
\textbf{\textup{(QCQP (\ref{eq:QCQP}): Closed-Form Solution)}}\label{thm:Solution}
Suppose that Assumptions \ref{AssumptionMain} and \ref{AssumptionMain2}
are in effect. Then, an optimal solution to problem (\ref{eq:QCQP})
may be expressed as (with slight abuse of notation)
\begin{equation}
\boxed{\widehat{\boldsymbol{X}}_{*}(\mu_{*})\equiv\dfrac{\mathbb{E}\big\{\hspace{-1pt}\boldsymbol{X}|\boldsymbol{Y}\big\}+\mu_{*}\big(\mathbb{E}\big\{\hspace{-1pt}\Vert\boldsymbol{X}\Vert_{2}^{2}\boldsymbol{X}|\boldsymbol{Y}\big\}\hspace{-0.5pt}-\mathbb{E}\big\{\hspace{-1pt}\Vert\boldsymbol{X}\Vert_{2}^{2}|\boldsymbol{Y}\big\}\mathbb{E}\{\boldsymbol{X}|\boldsymbol{Y}\}\big)}{\boldsymbol{I}+2\mu_{*}\boldsymbol{\Sigma}_{\boldsymbol{X}|\boldsymbol{Y}}},}\label{eq:ClosedForm}
\end{equation}
with $\widehat{\boldsymbol{X}}_{*}(\mu_{*})\in\arg\min_{\widehat{\boldsymbol{X}}\in{\cal F}_{2|\mathscr{Y}}^{M}}\mathpzc{L}\big(\widehat{\boldsymbol{X}},\mu_{*}\big)$,
and where $\mu_{*}\equiv\mu_{*}(\varepsilon)\in\mathbb{R}_{+}$ is
an optimal solution to the concave dual problem 
\begin{flalign}
\hspace{-1pt}\hspace{-1pt}\hspace{-1pt}\hspace{-1pt}\hspace{-1pt}\sup_{\mu\ge0}\mathpzc{D}(\mu) & \hspace{-1pt}\equiv\hspace{-1pt}\sup_{\mu\ge0}\inf_{\widehat{\boldsymbol{X}}\in{\cal F}_{2|\mathscr{Y}}^{M}}\mathpzc{L}\big(\widehat{\boldsymbol{X}},\mu\big)\nonumber \\
 & \hspace{-1pt}\equiv\hspace{-1pt}\hspace{-1pt}\dfrac{1}{2}\mathbb{E}\big\{\Vert\boldsymbol{X}\Vert_{2}^{2}\hspace{-1pt}\big\}+\dfrac{1}{4}\sup_{\mu\ge0}\Big\{\mu\mathbb{E}\big\{\mathbb{V}_{\boldsymbol{Y}}\{\Vert\boldsymbol{X}\Vert_{2}^{2}\}\hspace{-1pt}\big\}\nonumber \\
 & \quad-2\mathbb{E}\big\{\widehat{\boldsymbol{X}}_{*}^{\boldsymbol{T}}(\mu)(\boldsymbol{I}+2\mu\boldsymbol{\Sigma}_{\boldsymbol{X}|\boldsymbol{Y}})\widehat{\boldsymbol{X}}_{*}(\mu)\hspace{-1pt}\big\}-\mu\varepsilon\Big\}.\hspace{-1pt}\hspace{-1pt}\hspace{-1pt}
\end{flalign}
Additionally, the optimal risk-aware filter $\widehat{\boldsymbol{X}}_{*}(\mu_{*})$
is unique, almost everywhere relative to ${\cal P}$. 
\end{thm}
\begin{proof}[Proof of Theorem \ref{thm:Solution}]
First, for every $\mu\in\mathbb{R}_{+}$, let us consider the determination
of the dual function $\mathpzc{D}$ through solving the problem
\begin{equation}
\begin{array}{rl}
\mathrm{minimize} & \mathpzc{L}\big(\widehat{\boldsymbol{X}},\mu\big)\\
\mathrm{subject\,to} & \widehat{\boldsymbol{X}}\in{\cal F}_{2|\mathscr{Y}}^{M}
\end{array}.\label{eq:LAG_1}
\end{equation}
Let us also define the possibly extended real-valued, random function
$\mathpzc{r}(\cdot,\boldsymbol{Y}):\mathbb{R}^{M}\times\Omega\rightarrow\overline{\mathbb{R}}$,
quadratic in its first argument, as
\begin{flalign}
\mathpzc{r}(\boldsymbol{x},\boldsymbol{Y}) & \triangleq\dfrac{1}{2}\boldsymbol{x}^{\boldsymbol{T}}\big(\boldsymbol{I}+2\mu\boldsymbol{\Sigma}_{\boldsymbol{X}|\boldsymbol{Y}}\big)\boldsymbol{x}\nonumber \\
 & \quad-\big(\mathbb{E}\big\{\hspace{-1pt}\boldsymbol{X}|\boldsymbol{Y}\big\}+\mu\big(\mathbb{E}\big\{\hspace{-1pt}\Vert\boldsymbol{X}\Vert_{2}^{2}\boldsymbol{X}|\boldsymbol{Y}\hspace{-1pt}\big\}-\mathbb{E}\big\{\hspace{-1pt}\Vert\boldsymbol{X}\Vert_{2}^{2}|\boldsymbol{Y}\hspace{-1pt}\big\}\mathbb{E}\{\boldsymbol{X}|\boldsymbol{Y}\}\big)\big)^{\boldsymbol{T}}\boldsymbol{x}.
\end{flalign}
Observe that the quadratic term $\boldsymbol{x}^{\boldsymbol{T}}\big(\boldsymbol{I}+2\mu\boldsymbol{\Sigma}_{\boldsymbol{X}|\boldsymbol{Y}}\big)\boldsymbol{x}$
is finite ${\cal P}_{\boldsymbol{Y}}$-almost everywhere; indeed,
for every $\boldsymbol{x}\in\mathbb{R}^{M}$, it is true that
\begin{flalign}
0\le\boldsymbol{x}^{\boldsymbol{T}}\boldsymbol{\Sigma}_{\boldsymbol{X}|\boldsymbol{Y}}\boldsymbol{x} & \le\Vert\boldsymbol{x}\Vert_{2}^{2}\lambda_{max}(\boldsymbol{\Sigma}_{\boldsymbol{X}|\boldsymbol{Y}})\nonumber \\
 & \le\Vert\boldsymbol{x}\Vert_{2}^{2}\mathrm{tr}(\boldsymbol{\Sigma}_{\boldsymbol{X}|\boldsymbol{Y}})\nonumber \\
 & \equiv\Vert\boldsymbol{x}\Vert_{2}^{2}\mathbb{E}\big\{\hspace{-1pt}\Vert\boldsymbol{X}-\mathbb{E}\{\boldsymbol{X}|\boldsymbol{Y}\}\Vert_{2}^{2}|\boldsymbol{Y}\hspace{-1pt}\big\}\nonumber \\
 & \equiv\Vert\boldsymbol{x}\Vert_{2}^{2}\big(\mathbb{E}\big\{\hspace{-1pt}\Vert\boldsymbol{X}\Vert_{2}^{2}|\boldsymbol{Y}\hspace{-1pt}\big\}-\Vert\mathbb{E}\{\boldsymbol{X}|\boldsymbol{Y}\}\Vert_{2}^{2}\big)\nonumber \\
 & \le\Vert\boldsymbol{x}\Vert_{2}^{2}\mathbb{E}\big\{\hspace{-1pt}\Vert\boldsymbol{X}\Vert_{2}^{2}|\boldsymbol{Y}\hspace{-1pt}\big\},
\end{flalign}
where
\begin{flalign}
0\le\big(\mathbb{E}\big\{\hspace{-1pt}\Vert\boldsymbol{X}\Vert_{2}^{2}\hspace{-1pt}\big\}\big)^{3} & \le\big(\mathbb{E}\big\{\hspace{-1pt}\Vert\boldsymbol{X}\Vert_{2}^{3}\hspace{-1pt}\big\}\big)^{2}\nonumber \\
 & \le\mathbb{E}\big\{\hspace{-1pt}\big(\mathbb{E}\big\{\hspace{-1pt}\Vert\boldsymbol{X}\Vert_{2}^{3}|\boldsymbol{Y}\hspace{-1pt}\big\}\big)^{2}\hspace{-1pt}\big\}\nonumber \\
 & <\infty\quad\implies\quad\mathbb{E}\big\{\hspace{-1pt}\Vert\boldsymbol{X}\Vert_{2}^{2}|\boldsymbol{Y}\hspace{-1pt}\big\}<\infty,\quad{\cal P}_{\boldsymbol{Y}}-a.e.
\end{flalign}
Therefore, due to our assumptions, the function $\mathpzc{r}(\cdot,\boldsymbol{Y})$
is trivially continuous and finite on $\mathbb{R}^{M}$ up to sets
of ${\cal P}_{\boldsymbol{Y}}$-measure zero, those being independent
of each choice of $\boldsymbol{x}\in\mathbb{R}^{M}$, on which $\mathpzc{r}(\cdot,\boldsymbol{Y})$
may be arbitrarily defined. Consequently, $\mathpzc{r}(\cdot,\boldsymbol{Y})$
has a real-valued version, and thus may be taken as Carath\'{e}odory
on $\mathbb{R}^{M}\times\Omega$ (\cite{ShapiroLectures_2ND}, p.
421). Equivalently, $\mathpzc{r}$ may also be taken as Carath\'{e}odory
on $\mathbb{R}^{M}\times\mathbb{R}^{N}$, jointly measurable relative
to the Borel $\sigma$-algebra $\mathscr{B}\big(\mathbb{R}^{M}\times\mathbb{R}^{N}\big)$.

Under the above considerations, and given that Assumptions \ref{AssumptionMain}
and \ref{AssumptionMain2} are in effect, we may drop additive terms
which do not depend on the decision $\widehat{\boldsymbol{X}}$ in
problem (\ref{eq:LAG_1}), resulting in the equivalent problem\vspace{-2bp}
\begin{equation}
\begin{array}{rl}
\mathrm{minimize} & \mathbb{E}\{\mathpzc{r}(\widehat{\boldsymbol{X}},\boldsymbol{Y})\}\\
\mathrm{subject\,to} & \widehat{\boldsymbol{X}}\in{\cal L}_{2|\mathscr{Y}}^{M}
\end{array},\label{eq:LAG_2}
\end{equation}
where expectation may be conveniently taken directly over the Borel
probability space $\big(\mathbb{R}^{N},\mathscr{B}\big(\mathbb{R}^{N}\big),{\cal P}_{\boldsymbol{Y}}\big)$.
Note that (\ref{eq:LAG_2}) is uniformly lower bounded over ${\cal L}_{2|\mathscr{Y}}^{M}$,
through the definition of the Lagrangian $\mathpzc{L}$, and also
that, trivially, there is at least one choice of $\widehat{\boldsymbol{X}}\in{\cal L}_{2|\mathscr{Y}}^{M}$
such that $\mathbb{E}\{\mathpzc{r}(\widehat{\boldsymbol{X}},\boldsymbol{Y})\}<\infty$,
say $\mathbb{E}\{\mathpzc{r}(\boldsymbol{0},\boldsymbol{Y})\}\equiv0$,
for $\widehat{\boldsymbol{X}}\equiv{\bf 0}$.

Problem (\ref{eq:LAG_2}) may now be solved in closed form via application
of the \textit{Interchangeability Principle} (\cite{ShapiroLectures_2ND},
Theorem 7.92, or \cite{Rockafellar2009VarAn}, Theorem 14.60), which
is a fundamental result in variational optimization. To avoid unnecessary
generalities, we state it here for completeness adapted to our setting,
as follows.\medskip{}

\noindent %
\noindent\begin{minipage}[t]{1\columnwidth}%
\begin{thm}
\textbf{\textup{(Interchangeability Principle \cite{ShapiroLectures_2ND,Rockafellar2009VarAn})}}\label{thm:IP}
Let $\mathpzc{f}:\mathbb{R}^{M}\times\mathbb{R}^{N}\rightarrow\mathbb{R}$
be Carath\'{e}odory, and fix $p\in[1,\infty]$. It is true that
\begin{equation}
{\textstyle \inf_{\widehat{\boldsymbol{X}}\in{\cal L}_{p|\mathscr{Y}}^{M}}}\,\mathbb{E}\big\{\mathpzc{f}(\widehat{\boldsymbol{X}},\boldsymbol{Y})\hspace{-1pt}\big\}\equiv\mathbb{E}\big\{{\textstyle \inf_{\boldsymbol{x}\in\mathbb{R}^{M}}}\,\mathpzc{f}(\boldsymbol{x},\boldsymbol{Y})\hspace{-1pt}\big\},\label{eq:IP}
\end{equation}
provided that the left-hand side of (\ref{eq:IP}) is less that $+\infty$.
If, additionally, either of the sides of (\ref{eq:IP}) is not $-\infty$,
it is also true that
\begin{flalign}
 & \widehat{\boldsymbol{X}}_{*}\hspace{-1pt}\in\hspace{-1pt}{\textstyle \arg\min_{\widehat{\boldsymbol{X}}\in{\cal L}_{p|\mathscr{Y}}^{M}}}\,\mathbb{E}\big\{\mathpzc{f}(\widehat{\boldsymbol{X}},\boldsymbol{Y})\big\}\\
\hspace{-1pt}\hspace{-1pt}\iff\, & \widehat{\boldsymbol{X}}_{*}\hspace{-1pt}\in\hspace{-1pt}{\textstyle \arg\min_{\boldsymbol{x}\in\mathbb{R}^{M}}}\,\mathpzc{f}(\boldsymbol{x},\boldsymbol{Y}),\hspace{-1pt}\hspace{-1pt}\text{ for }{\cal P}_{\boldsymbol{Y}}\text{-almost all }\boldsymbol{Y},\hspace{-1pt}\hspace{-1pt}\text{ and }\widehat{\boldsymbol{X}}_{*}\hspace{-1pt}\in\hspace{-1pt}{\cal L}_{p|\mathscr{Y}}^{M}.
\end{flalign}
\end{thm}
\end{minipage}\vspace{8bp}

Let us apply Theorem \ref{thm:IP} to the variational problem (\ref{eq:LAG_2}),
for $p\equiv2$. Then, the (\ref{eq:LAG_2}) may be exchanged by the
pointwise (over constants) quadratic problem
\begin{equation}
{\textstyle \inf_{\boldsymbol{x}\in\mathbb{R}^{M}}}\,\mathpzc{r}(\boldsymbol{x},\boldsymbol{Y}),\label{eq:LAG_3}
\end{equation}
whose unique solution is, for every $\mu\in\mathbb{R}_{+}$ and for
every value of $\boldsymbol{Y}\in\mathbb{R}^{N}$,\vspace{-4.5bp}
\begin{multline}
\quad\quad\quad\quad\;\,\widehat{\boldsymbol{X}}_{*}(\mu)\equiv(\boldsymbol{I}+2\mu\boldsymbol{\Sigma}_{\boldsymbol{X}|\boldsymbol{Y}})^{-1}\big(\mathbb{E}\big\{\hspace{-1pt}\boldsymbol{X}|\boldsymbol{Y}\big\}\\
+\mu\big(\mathbb{E}\big\{\hspace{-1pt}\Vert\boldsymbol{X}\Vert_{2}^{2}\boldsymbol{X}|\boldsymbol{Y}\big\}\hspace{-1pt}\hspace{-1pt}-\hspace{-1pt}\mathbb{E}\big\{\hspace{-1pt}\Vert\boldsymbol{X}\Vert_{2}^{2}|\boldsymbol{Y}\big\}\mathbb{E}\{\boldsymbol{X}|\boldsymbol{Y}\}\big)\hspace{-1pt}\big),\quad\quad\quad\quad\label{eq:SOL_1}
\end{multline}
which is precisely the expression claimed in Theorem \ref{thm:Solution},
for a generic $\mu$. In order to show that (\ref{eq:SOL_1}) is a
solution of problem (\ref{eq:LAG_2}) and, in turn, (\ref{eq:LAG_1}),
we also have to verify that $\widehat{\boldsymbol{X}}_{*}(\mu)\in{\cal L}_{2|\mathscr{Y}}^{M}$.
We may write, by Cauchy-Schwarz (note that $\Vert(\boldsymbol{I}+2\mu\boldsymbol{\Sigma}_{\boldsymbol{X}|\boldsymbol{Y}})^{-1}\Vert_{2}\le1$),
the triangle inequality, and Jensen,
\begin{flalign}
\Vert\widehat{\boldsymbol{X}}_{*}(\mu)\Vert_{2} & \le\Vert\mathbb{E}\{\boldsymbol{X}|\boldsymbol{Y}\}+\mu\big(\mathbb{E}\big\{\hspace{-1pt}\Vert\boldsymbol{X}\Vert_{2}^{2}\boldsymbol{X}|\boldsymbol{Y}\big\}\hspace{-1pt}\hspace{-1pt}-\hspace{-1pt}\mathbb{E}\big\{\hspace{-1pt}\Vert\boldsymbol{X}\Vert_{2}^{2}|\boldsymbol{Y}\big\}\mathbb{E}\{\boldsymbol{X}|\boldsymbol{Y}\}\big)\Vert_{2}\nonumber \\
 & \le\Vert\mathbb{E}\{\boldsymbol{X}|\boldsymbol{Y}\}\Vert_{2}+\mu\Vert\mathbb{E}\big\{\hspace{-1pt}\Vert\boldsymbol{X}\Vert_{2}^{2}\boldsymbol{X}|\boldsymbol{Y}\big\}\Vert_{2}+\mu\Vert\mathbb{E}\big\{\hspace{-1pt}\Vert\boldsymbol{X}\Vert_{2}^{2}|\boldsymbol{Y}\big\}\mathbb{E}\{\boldsymbol{X}|\boldsymbol{Y}\}\Vert_{2}\nonumber \\
 & \le\mathbb{E}\{\Vert\boldsymbol{X}\Vert_{2}|\boldsymbol{Y}\}+\mu\mathbb{E}\big\{\hspace{-1pt}\Vert\boldsymbol{X}\Vert_{2}^{3}|\boldsymbol{Y}\big\}+\mu\mathbb{E}\big\{\hspace{-1pt}\Vert\boldsymbol{X}\Vert_{2}^{2}|\boldsymbol{Y}\big\}\mathbb{E}\{\Vert\boldsymbol{X}\Vert_{2}|\boldsymbol{Y}\},\label{eq:VER_1}
\end{flalign}
and we are done, since we have already shown that all three terms
in the right-hand side of (\ref{eq:VER_1}) are in ${\cal L}_{2|\mathscr{Y}}^{1}$.

The final step in the proof of Theorem \ref{thm:Solution} is to exploit
strong duality of the QCQP (\ref{eq:QCQP}) by invoking Theorem \ref{thm:ZDG_Var}.
Indeed, it follows that the optimal value of the primal problem (\ref{eq:QCQP})
coincides with that of the dual problem (\ref{eq:DUAL}), which may
be expressed as
\begin{flalign}
 & \hspace{-0.5pt}\sup_{\mu\ge0}\mathpzc{D}(\mu)\nonumber \\
 & \equiv\hspace{-1pt}\sup_{\mu\ge0}\inf_{\widehat{\boldsymbol{X}}\in{\cal F}_{2|\mathscr{Y}}^{M}}\mathpzc{L}\big(\widehat{\boldsymbol{X}},\mu\big)\nonumber \\
 & \equiv\hspace{-1pt}\sup_{\mu\ge0}\bigg\{\hspace{-1pt}\dfrac{1}{2}\mathbb{E}\big\{\Vert\boldsymbol{X}\Vert_{2}^{2}\hspace{-1pt}\big\}+\dfrac{1}{4}\mu\mathbb{E}\big\{\mathbb{V}_{\boldsymbol{Y}}\{\Vert\boldsymbol{X}\Vert_{2}^{2}\}\hspace{-1pt}\big\}\nonumber \\
 & \;\,+\mathbb{E}\Big\{\dfrac{1}{2}\Vert\widehat{\boldsymbol{X}}_{*}(\mu)\Vert_{2}^{2}+\mu\widehat{\boldsymbol{X}}_{*}(\mu)^{\boldsymbol{T}}\boldsymbol{\Sigma}_{\boldsymbol{X}|\boldsymbol{Y}}\widehat{\boldsymbol{X}}_{*}(\mu)\nonumber \\
 & \;\,\;\,-(\mathbb{E}\big\{\hspace{-1pt}\boldsymbol{X}|\boldsymbol{Y}\big\})^{\boldsymbol{T}}\widehat{\boldsymbol{X}}_{*}(\mu)\hspace{-1pt}-\hspace{-1pt}\mu\big(\mathbb{E}\big\{\hspace{-1pt}\Vert\boldsymbol{X}\Vert_{2}^{2}\boldsymbol{X}|\boldsymbol{Y}\hspace{-1pt}\big\}\hspace{-1pt}-\hspace{-1pt}\mathbb{E}\big\{\hspace{-1pt}\Vert\boldsymbol{X}\Vert_{2}^{2}|\boldsymbol{Y}\hspace{-1pt}\big\}\mathbb{E}\{\boldsymbol{X}|\boldsymbol{Y}\}\big)^{\boldsymbol{T}}\widehat{\boldsymbol{X}}_{*}(\mu)\hspace{-1pt}\Big\}\nonumber \\
 & \;\,\;\,\;\,-\dfrac{\mu\varepsilon}{4}\bigg\}\nonumber \\
 & \equiv\hspace{-1pt}\hspace{-1pt}\dfrac{1}{2}\mathbb{E}\big\{\Vert\boldsymbol{X}\Vert_{2}^{2}\hspace{-1pt}\big\}+\sup_{\mu\ge0}\bigg\{\dfrac{1}{4}\mu\mathbb{E}\big\{\mathbb{V}_{\boldsymbol{Y}}\{\Vert\boldsymbol{X}\Vert_{2}^{2}\}\hspace{-1pt}\big\}\nonumber \\
 & \;\,+\mathbb{E}\Big\{\dfrac{1}{2}\widehat{\boldsymbol{X}}_{*}(\mu)^{\boldsymbol{T}}(\boldsymbol{I}+2\mu\boldsymbol{\Sigma}_{\boldsymbol{X}|\boldsymbol{Y}})\widehat{\boldsymbol{X}}_{*}(\mu)\nonumber \\
 & \;\,\;\,-\big(\mathbb{E}\big\{\hspace{-1pt}\boldsymbol{X}|\boldsymbol{Y}\big\}+\mu\big(\mathbb{E}\big\{\hspace{-1pt}\Vert\boldsymbol{X}\Vert_{2}^{2}\boldsymbol{X}|\boldsymbol{Y}\hspace{-1pt}\big\}-\mathbb{E}\big\{\hspace{-1pt}\Vert\boldsymbol{X}\Vert_{2}^{2}|\boldsymbol{Y}\hspace{-1pt}\big\}\mathbb{E}\{\boldsymbol{X}|\boldsymbol{Y}\}\big)\big)^{\boldsymbol{T}}\widehat{\boldsymbol{X}}_{*}(\mu)\hspace{-1pt}\Big\}-\dfrac{\mu\varepsilon}{4}\bigg\}\nonumber \\
 & \equiv\hspace{-1pt}\hspace{-1pt}\dfrac{1}{2}\mathbb{E}\big\{\Vert\boldsymbol{X}\Vert_{2}^{2}\hspace{-1pt}\big\}+\sup_{\mu\ge0}\bigg\{\dfrac{1}{4}\mu\mathbb{E}\big\{\mathbb{V}_{\boldsymbol{Y}}\{\Vert\boldsymbol{X}\Vert_{2}^{2}\}\hspace{-1pt}\big\}\nonumber \\
 & \;\,+\mathbb{E}\Big\{\dfrac{1}{2}\widehat{\boldsymbol{X}}_{*}(\mu)^{\boldsymbol{T}}(\boldsymbol{I}+2\mu\boldsymbol{\Sigma}_{\boldsymbol{X}|\boldsymbol{Y}})\widehat{\boldsymbol{X}}_{*}(\mu)-\widehat{\boldsymbol{X}}_{*}(\mu)^{\boldsymbol{T}}(\boldsymbol{I}+2\mu\boldsymbol{\Sigma}_{\boldsymbol{X}|\boldsymbol{Y}})\widehat{\boldsymbol{X}}_{*}(\mu)\hspace{-1pt}\Big\}-\dfrac{\mu\varepsilon}{4}\bigg\}\nonumber \\
 & \hspace{-1pt}\equiv\hspace{-1pt}\hspace{-1pt}\dfrac{1}{2}\mathbb{E}\big\{\Vert\boldsymbol{X}\Vert_{2}^{2}\hspace{-1pt}\big\}+\dfrac{1}{4}\sup_{\mu\ge0}\Big\{\mu\mathbb{E}\big\{\mathbb{V}_{\boldsymbol{Y}}\{\Vert\boldsymbol{X}\Vert_{2}^{2}\}\hspace{-1pt}\big\}\nonumber \\
 & \;\,-2\mathbb{E}\big\{\widehat{\boldsymbol{X}}_{*}^{\boldsymbol{T}}(\mu)(\boldsymbol{I}+2\mu\boldsymbol{\Sigma}_{\boldsymbol{X}|\boldsymbol{Y}})\widehat{\boldsymbol{X}}_{*}(\mu)\hspace{-1pt}\big\}-\mu\varepsilon\Big\}.
\end{flalign}
Finally, let $\mu_{*}\ge0$ be a maximizer of $\mathpzc{D}$ over
$\mathbb{R}_{+}$ such that $\mathpzc{D}^{*}\equiv\mathpzc{P}^{*}<\infty$,
and suppose that $\widetilde{\boldsymbol{X}}_{*}$ is primal optimal
for (\ref{eq:QCQP}). By strong duality, it is true that
\begin{equation}
\widetilde{\boldsymbol{X}}_{*}\equiv\widetilde{\boldsymbol{X}}_{*}(\mu_{*})\in{\textstyle \arg\min_{\widehat{\boldsymbol{X}}\in{\cal F}_{2|\mathscr{Y}}^{M}}}\mathpzc{L}\big(\widehat{\boldsymbol{X}},\mu_{*}\big).\label{eq:SOL_2}
\end{equation}
By uniqueness of $\widehat{\boldsymbol{X}}_{*}(\mu_{*})$ in (\ref{eq:SOL_1})
(pointwise in $\boldsymbol{Y}$), all members of the possibly infinite
set of optimal solutions of (\ref{eq:LAG_1}) (for $\mu\equiv\mu_{*}$)
in (\ref{eq:SOL_2}) differ at most on sets of measure zero, and result
in exactly the same values for both the objective and constraints
of (\ref{eq:QCQP}). Therefore, all such optimal solutions to (\ref{eq:LAG_1})
are also optimal for (\ref{eq:QCQP}) and, in particular, $\widehat{\boldsymbol{X}}_{*}(\mu_{*})$
is one of them. Enough said.
\end{proof}
Theorem \ref{thm:Solution} completely solves problem (\ref{eq:Base_Problem-L2})
by providing a closed-form expression for the risk-aware MMSE estimator
$\widehat{\boldsymbol{X}}_{*}$, defined in terms of the dual optimal
solution $\mu_{*}$ (a number). The latter \textit{always} exists,
thanks to Theorem \ref{thm:ZDG_Var}, and may be computed by leveraging
our knowledge of the distribution ${\cal P}_{(\boldsymbol{X},\boldsymbol{Y})}$
and via either some gradient-based method, or even empirically. Note
that the dual function $\mathpzc{D}$ is merely a concave function
on the positive line.

The fact that a closed-form optimal solution to problem (\ref{eq:QCQP})
exists is remarkable, and it provides insight into the intrinsic structure
of \textit{constrained} Bayesian risk-aware estimation, also enabling
an explicit comparison of the optimal risk-aware filter $\widehat{\boldsymbol{X}}_{*}$
with its risk-neutral counterpart. Indeed, by looking at the explicit
form of the optimal risk-aware filter $\widehat{\boldsymbol{X}}_{*}$,
we readily see that it is a function involving the MMSE estimator
$\mathbb{E}\big\{\hspace{-1pt}\boldsymbol{X}|\boldsymbol{Y}\big\}$,
its predictive covariance matrix $\boldsymbol{\Sigma}_{\boldsymbol{X}|\boldsymbol{Y}}$,
as well as the second and third order filters $\mathbb{E}\big\{\hspace{-1pt}\Vert\boldsymbol{X}\Vert_{2}^{2}|\boldsymbol{Y}\big\}$
and $\mathbb{E}\big\{\hspace{-1pt}\Vert\boldsymbol{X}\Vert_{2}^{2}\boldsymbol{X}|\boldsymbol{Y}\big\}$.
All these quantities are elementary and, in principle, they can be
evaluated by utilizing a single observation of $\boldsymbol{Y}$,
and by exploiting our knowledge of the conditional measure ${\cal P}_{\boldsymbol{X}|\boldsymbol{Y}}$,
just as in risk-neutral MMSE estimation.

Also, we see that $\widehat{\boldsymbol{X}}_{*}$ may be regarded
as a \textit{biased MMSE estimator}, drawing parallels to \textit{James-Stein
estimators}, in another statistical context. Through the effect of
bias, while James-Stein estimators achieve lower mean squared error,
$\widehat{\boldsymbol{X}}_{*}$ achieves lower risk. Therefore, \textit{optimal
risk aversion, in the sense of problem (\ref{eq:Base_Problem-V}),
may be interpreted as the result of bias injection in MMSE estimators}.


Additionally, we observe that the solution is \textit{regularized},
in the sense that the term $2\mu_{*}\boldsymbol{\Sigma}_{\boldsymbol{X}|\boldsymbol{Y}}$
is diagonally loaded with an identity matrix; as a result, $\widehat{\boldsymbol{X}}_{*}$
is always well-defined and numerically stable. In fact, whenever $\mu_{*}\equiv0$
(depending on the magnitude of the tolerance $\varepsilon$), it follows
that $\widehat{\boldsymbol{X}}_{*}\equiv\mathbb{E}\big\{\hspace{-1pt}\boldsymbol{X}|\boldsymbol{Y}\big\}$.
But this is not the only case where the two estimators turn out to
be the same. The next result confirms that there exists a certain
family of models for which risk-neutral and risk-aware MMSE estimation
actually coincide; in such cases, posing (\ref{eq:Base_Problem-L2})
is redundant.
\begin{thm}
\textbf{\textup{(When do Risk-Neutral/Aware Filters Coincide?)}}\label{thm:RiskNeutral}
Suppose that the conditional measure ${\cal P}_{\boldsymbol{X}|\boldsymbol{Y}}$
is such that 
\begin{equation}
\mathbb{E}\big\{\hspace{-1pt}\hspace{-1pt}\big(X_{i}-\mathbb{E}\{X_{i}|\boldsymbol{Y}\}\big)^{2}\big(\boldsymbol{X}-\mathbb{E}\{\boldsymbol{X}|\boldsymbol{Y}\}\big)|\boldsymbol{Y}\big\}\hspace{-1pt}\equiv{\bf 0},\;\forall i\in\mathbb{N}_{M}^{+}.
\end{equation}
Then, under Assumption \ref{AssumptionMain} and in the notation of
Theorem \ref{thm:Solution}, it is true that, for every $\mu\ge0$,
$\widehat{\boldsymbol{X}}_{*}(\mu)\equiv\mathbb{E}\{\boldsymbol{X}|\boldsymbol{Y}\}$.
In other words, under both Assumptions \ref{AssumptionMain} and \ref{AssumptionMain2},
risk-neutral MMSE estimation is also risk-aware, for every qualifying
value of $\varepsilon>0$. In particular, this is the case whenever
${\cal P}_{\boldsymbol{X}|\boldsymbol{Y}}$ is joint Gaussian.
\end{thm}
\begin{proof}[Proof of Theorem \ref{thm:RiskNeutral}]
First, it is a simple exercise to show that (note that all involved
conditional expectations in the expression above assume finite values
due to Assumption \ref{AssumptionMain})
\begin{align}
{\bf 0} & \equiv\mathbb{E}\big\{\hspace{-1pt}\hspace{-1pt}\big(X_{i}-\mathbb{E}\big\{\hspace{-1pt}X_{i}|\boldsymbol{Y}\big\}\big)^{2}\big(\boldsymbol{X}-\mathbb{E}\big\{\hspace{-1pt}\boldsymbol{X}|\boldsymbol{Y}\big\}\big)|\boldsymbol{Y}\big\}\nonumber \\
 & \equiv\mathbb{E}\big\{\hspace{-1pt}\hspace{-1pt}\big(X_{i}^{2}\boldsymbol{X}\big)|\boldsymbol{Y}\big\}-\mathbb{E}\big\{\hspace{-1pt}\boldsymbol{X}|\boldsymbol{Y}\big\}\mathbb{E}\big\{\hspace{-1pt}\hspace{-1pt}X_{i}^{2}|\boldsymbol{Y}\big\}\nonumber \\
 & \quad-2\mathbb{E}\big\{\hspace{-1pt}X_{i}|\boldsymbol{Y}\big\}\mathbb{E}\big\{\hspace{-1pt}X_{i}\boldsymbol{X}|\boldsymbol{Y}\big\}+2(\mathbb{E}\big\{\hspace{-1pt}X_{i}|\boldsymbol{Y}\big\})^{2}\mathbb{E}\big\{\hspace{-1pt}\boldsymbol{X}|\boldsymbol{Y}\big\}
\end{align}
which of course implies that
\begin{flalign}
\mathbb{E}\big\{\hspace{-1pt}\hspace{-1pt}\big(X_{i}^{2}\boldsymbol{X}\big)|\boldsymbol{Y}\big\} & \equiv\mathbb{E}\{\boldsymbol{X}|\boldsymbol{Y}\}\mathbb{E}\big\{\hspace{-1pt}\hspace{-1pt}X_{i}^{2}|\boldsymbol{Y}\big\}\nonumber \\
 & \quad+2\mathbb{E}\{X_{i}|\boldsymbol{Y}\}\big(\mathbb{E}\{X_{i}\boldsymbol{X}|\boldsymbol{Y}\}-\mathbb{E}\{X_{i}|\boldsymbol{Y}\}\mathbb{E}\{\boldsymbol{X}|\boldsymbol{Y}\}\big).\label{eq:KEY_Coincide}
\end{flalign}
Therefore, we have
\begin{flalign}
 & \hspace{-1pt}\hspace{-1pt}\hspace{-1pt}\hspace{-1pt}\hspace{-1pt}\hspace{-1pt}\hspace{-1pt}\hspace{-1pt}\hspace{-1pt}\hspace{-1pt}\hspace{-1pt}\hspace{-1pt}\mathbb{E}\big\{\hspace{-1pt}\Vert\boldsymbol{X}\Vert_{2}^{2}\boldsymbol{X}|\boldsymbol{Y}\big\}\hspace{-1pt}\hspace{-1pt}\hspace{-1pt}-\hspace{-1pt}\hspace{-1pt}\mathbb{E}\big\{\hspace{-1pt}\Vert\boldsymbol{X}\Vert_{2}^{2}|\boldsymbol{Y}\big\}\mathbb{E}\{\boldsymbol{X}|\boldsymbol{Y}\}\nonumber \\
 & \equiv\sum_{i\in\mathbb{N}_{N}^{+}}\mathbb{E}\big\{\hspace{-1pt}X_{i}^{2}\boldsymbol{X}|\boldsymbol{Y}\big\}-\mathbb{E}\big\{\hspace{-1pt}X_{i}^{2}|\boldsymbol{Y}\big\}\mathbb{E}\{\boldsymbol{X}|\boldsymbol{Y}\}\nonumber \\
 & \equiv2\sum_{i\in\mathbb{N}_{N}^{+}}\mathbb{E}\{X_{i}|\boldsymbol{Y}\}(\mathbb{E}\{X_{i}\boldsymbol{X}|\boldsymbol{Y}\}-\mathbb{E}\{X_{i}|\boldsymbol{Y}\}\mathbb{E}\{\boldsymbol{X}|\boldsymbol{Y}\})\nonumber \\
 & \equiv2\boldsymbol{\Sigma}_{\boldsymbol{X}|\boldsymbol{Y}}\mathbb{E}\{\boldsymbol{X}|\boldsymbol{Y}\},
\end{flalign}
which in turn implies that, for every $\mu\ge0$,
\begin{flalign}
\widehat{\boldsymbol{X}}_{*}(\mu) & \equiv\dfrac{\mathbb{E}\big\{\hspace{-1pt}\boldsymbol{X}|\boldsymbol{Y}\big\}+\mu\big(\mathbb{E}\big\{\hspace{-1pt}\Vert\boldsymbol{X}\Vert_{2}^{2}\boldsymbol{X}|\boldsymbol{Y}\big\}\hspace{-0.5pt}-\mathbb{E}\big\{\hspace{-1pt}\Vert\boldsymbol{X}\Vert_{2}^{2}|\boldsymbol{Y}\big\}\mathbb{E}\{\boldsymbol{X}|\boldsymbol{Y}\}\big)}{\boldsymbol{I}+2\mu\boldsymbol{\Sigma}_{\boldsymbol{X}|\boldsymbol{Y}}}\nonumber \\
 & \equiv\dfrac{\mathbb{E}\big\{\hspace{-1pt}\boldsymbol{X}|\boldsymbol{Y}\big\}+2\mu\boldsymbol{\Sigma}_{\boldsymbol{X}|\boldsymbol{Y}}\mathbb{E}\{\boldsymbol{X}|\boldsymbol{Y}\}}{\boldsymbol{I}+2\mu\boldsymbol{\Sigma}_{\boldsymbol{X}|\boldsymbol{Y}}}\nonumber \\
 & \equiv\dfrac{(\boldsymbol{I}+2\mu\boldsymbol{\Sigma}_{\boldsymbol{X}|\boldsymbol{Y}})\mathbb{E}\{\boldsymbol{X}|\boldsymbol{Y}\}}{\boldsymbol{I}+2\mu\boldsymbol{\Sigma}_{\boldsymbol{X}|\boldsymbol{Y}}}\nonumber \\
 & \equiv(\boldsymbol{I}+2\mu\boldsymbol{\Sigma}_{\boldsymbol{X}|\boldsymbol{Y}})^{-1}\big(\boldsymbol{I}+2\mu\boldsymbol{\Sigma}_{\boldsymbol{X}|\boldsymbol{Y}}\big)\mathbb{E}\{\boldsymbol{X}|\boldsymbol{Y}\}\nonumber \\
 & \equiv\mathbb{E}\{\boldsymbol{X}|\boldsymbol{Y}\},\quad{\cal P}-a.e.
\end{flalign}
The fact that (\ref{eq:KEY_Coincide}) is true when ${\cal P}_{\boldsymbol{X}|\boldsymbol{Y}}$
is multivariate Gaussian follows from the straightforward application
of Stein's Lemma on all pairs of jointly Gaussian random variables
$(X_{i},X_{j})$, $(i,j)\in\mathbb{N}_{M}^{+}\times\mathbb{N}_{M}^{+}$,
conditioned on $\boldsymbol{Y}$.
\end{proof}
In the next section, we put the risk-aware MMSE estimator to work,
as well as numerically evaluate its performance in comparison with
that of the usual, risk-neutral MMSE estimator.

\section{Numerical Simulations}

\begin{figure}
\centering \includesvg[width=310pt]{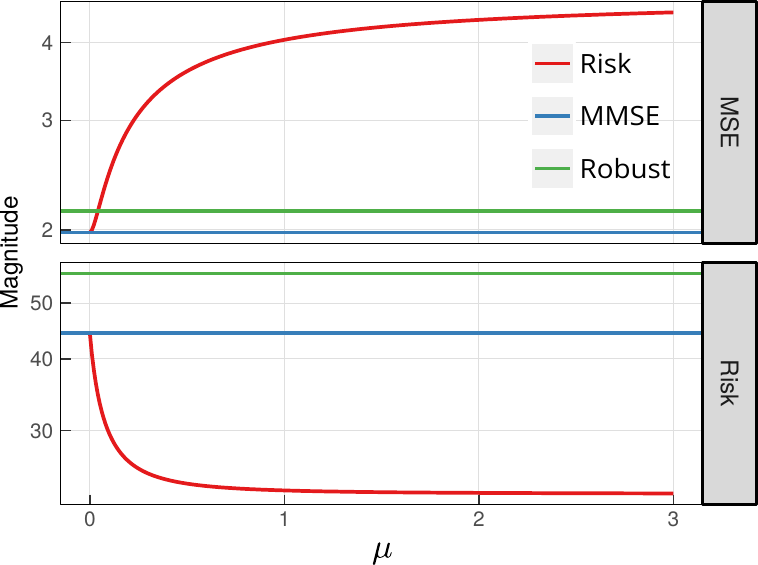} \caption{\label{F:fig2}Mean squared error and risk for different values of~$\mu$
in the state-dependent noise scenario.}
\end{figure}
We evaluate the behavior of the estimator in (\ref{eq:ClosedForm})
in two different scenarios. The first consists of the problem of estimating
an exponentially distributed hidden state $X$, $\mathbb{E}\{X\}=2$,
from the observation $Y=X+v$, where $v$ is a zero-mean Gaussian
random variable independent of~$X$ whose variance is given by $\mathbb{E}\{v^{2}\}=9X^{2}$;
in this case, $v$ constitutes a state-dependent noise. In the second
scenario, the goal is to jointly estimate the random vector $X=\left[z\ \ h\right]^{T}$
from the observation $Y=hz+w$, where $z$ is a zero-mean Gaussian
random variable with variance $\mathbb{E}\{z^{2}\}=2$, $h$ has a
Rayleigh distribution with rate $2$, and $w$ is a zero-mean Gaussian
noise with variance $\mathbb{E}\{w^{2}\}=10^{-1}$. This scenario
is prototypical for estimation problems in communications, where $z$
is the signal of interest and $h$ represents the channel fading.
Throughout the simulations, we also show results for the risk-neutral
MMSE estimator and the Minimum Mean Absolute Error (MMAE) estimator,
or, equivalently, the conditional median relative to the respective
observations, the latter being used as an example of a robust location
parameter estimator.
\begin{figure}
\centering \includesvg[width=328pt]{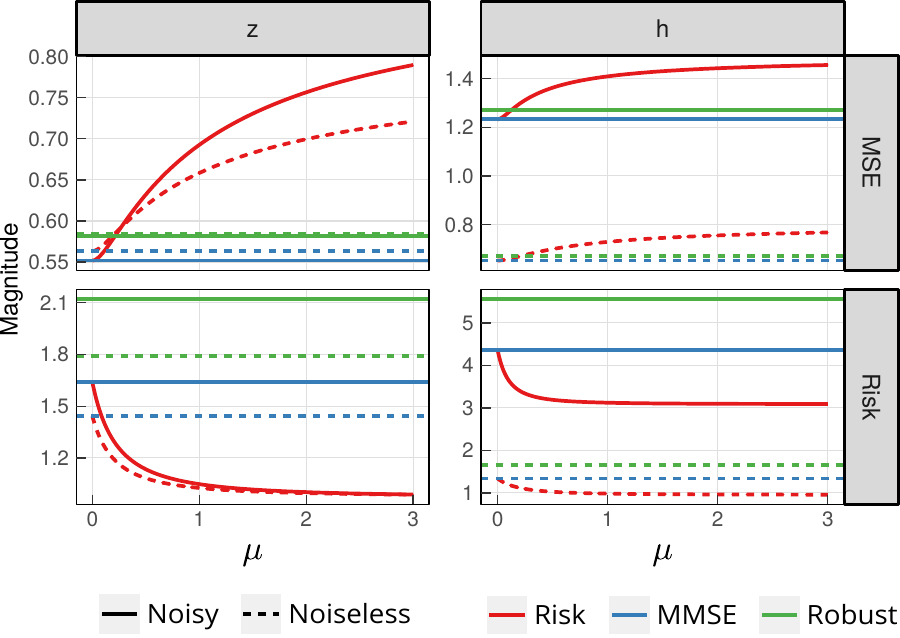} \caption{\label{F:fig3}Mean squared error and risk for different values of
$\mu$ in the communication scenario.}
\end{figure}

In Fig. 1, we saw that the risk-aware estimator yields larger estimates
than the MMSE estimator, in order to account for the certain statistical
ambiguities of the state-dependent noise model. Though this difference
may seem extreme in some instances, e.g., for small values of $Y$
(as in Fig. 1), it is in fact quite effective in reducing the conditional
variance. Indeed, for $Y=0.1$, the risk-aware estimator in Fig. 1
optimally reduces the (conditional) risk by approximately $26\%$
as compared to the risk of the risk-neutral estimator, and this is
achieved by sacrificing average performance, also by a factor of $26\%$.
Of course, this is only one of the operation points of the risk-aware
estimator. In Fig. \ref{F:fig2}, we show results for different values
of $\mu$, where we average over the distribution of $Y$. Observe
that the risk-aware estimator obtained using the constrained optimization
problem (\ref{eq:Base_Problem-V}) achieves a sharp trade-off between
average performance (that is, mean squared error) and risk, which
can be tuned according to the needs of the application. Additionally,
note that the decrease in risk is considerably faster than the increase
in MSE.

Interestingly, a similar phenomenon is observed in the communication
scenario (Fig. \ref{F:fig3}). Again, the risk-aware estimator displays
a much faster initial rate of decrease with respect to $\mu$ than
the rate at which the MSE increases. This is more pronounced in the
estimation of the component $z$, for which the risk-aware estimator
can provide reductions of almost $60\%$ in risk for a $35\%$ increase
in average squared error. Note that, as per Theorem \ref{thm:RiskNeutral},
the Gaussian noise has indeed no bearing on risk-awareness, as evidenced
by the performance in the noiseless case, i.e., for $w=0$ (dashed
lines). To achieve the behavior of Fig. 3, the risk-aware estimator
overestimates both $z$ and $h$ as compared to the MMSE estimator,
as illustrated in Fig. \ref{F:fig4}. In fact, for small values of
$Y$, the former hedges against the event of a deep fade ($h\approx0$)
by maintaining its estimates for $z$ away from zero.
\begin{figure}
\centering \includesvg[width=310pt]{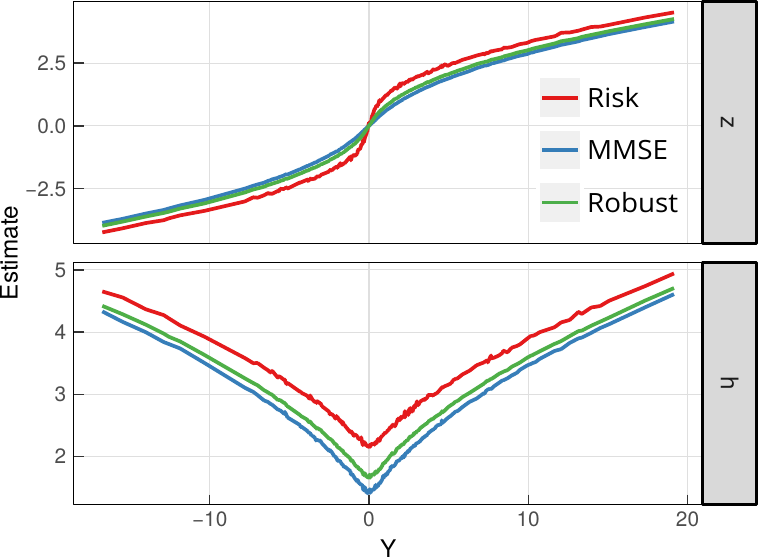} \caption{\label{F:fig4}Risk-aware, MMSE, and robust estimates of $z$ and
$h$ in the communication scenario for different values of $Y$.}
\end{figure}

\section{Conclusions}

We derived a risk-aware MMSE estimator that accounts for statistical
model volatility by hedging against extreme losses. We did so by formulating
a Bayesian risk-aware MMSE estimation problem that minimizes squared
errors on average, subject to an explicit tolerance on their expected
predictive variance. We then showed that this problem admits a analytical
solution under mild moment boundedness assumptions, and results in
a risk-aware, biased nonlinear MMSE estimator. The effectiveness of
our approach was confirmed via several numerical simulations. Future
work includes further analysis of the statistical properties of the
proposed estimator, as well as study of other constrained risk-aware
formulations.


 \bibliographystyle{IEEEbib}
\bibliography{library_fixed}

\end{document}